\documentclass[10pt]{amsart}

\usepackage{amsthm, amsfonts, amssymb, color}
 \usepackage{mathrsfs}
\usepackage{amsmath}
 \usepackage{amstext, amsxtra}
  \usepackage{txfonts}
 \usepackage[colorlinks, linkcolor=black, citecolor=blue, pagebackref, hypertexnames=false]{hyperref}
\usepackage{graphicx}
\usepackage{overpic}

 \allowdisplaybreaks
\newtheorem{cor}{\hskip\parindent {Corollary}}[section]

\newtheorem{theorem}{Theorem}[section]

\newtheorem{lemma}[theorem]{Lemma}
\newtheorem{definition}[theorem]{Definition}

\newtheorem{remark}[theorem]{Remark}

\title[Navier-Stokes equations, symmetry and Herz spaces]{
Navier-Stokes equations, symmetric and uniform analytic solutions in phase space}
\author{Qixiang Yang}

\address {Qixiang Yang, School of Mathematics and Statistics, Wuhan University, Wuhan, 430072, P. R. China}
\email{qxyang@whu.edu.cn}
\date{\today}
\subjclass[2010]{ 35Q30; 76D03; 42B35}
\keywords{Navier-Stokes equations;  Symmetry about the component of velocity field;
Symmetry and anti-symmetry for the independent variable of velocity field;
Uniform analyticity; Fourier-Herz space.}

\begin{document}

\maketitle
\begin{abstract}
For incompressible Navier-Stokes equations, Necas-Ruzicka-Sverak proved that
self-similar solution has to be zero in 1996.
Further, Yang-Yang-Wu find symmetry property plays an important role in the study of ill-posedness.
In this paper, we consider two types of symmetry property.
We search special symmetric and uniform analytic functions to approach the solution
and establish global uniform analytic and symmetric solution
with initial value in general symmetric Fourier-Herz space.
For two kinds of symmetry of initial data, we prove that the solution has also the same symmetric structure.
Further, we prove that the uniform analyticity is equivalent to the convolution inequality on Herz spaces.
By these ways, we can use symmetric and uniform analytic functions to approximate  the solution.
\end{abstract}


\section{Introduction}
\setcounter{equation}{0}
We consider the incompressible Navier-Stokes equations on the half-space $\mathbb{R}_{+}\times \mathbb{R}^{3}$,
\begin{equation}\label{1.1}
\left\{ \begin{aligned}
  &\partial_t u -\Delta u+u\cdot\nabla u-\nabla p=0,\;\;(t,x) \in \mathbb{R}_{+}\times \mathbb{R}^{3}, \\
  &\nabla\cdot u=0,\\
  u&(0, x)=u_{0}(x).
 \end{aligned}\right.
 \end{equation}
Since Kato-Fujita find mild solution for Navier-Stokes equations (\ref{1.1}), there are many works which considered the mild solution.
See  \cite{Canj, Canb, Ger, Gig, Kat, KatP, Koc, Lei, Lem, LiX, LiY1, LiY2,  Lin, Tay, Wu1, Wu2, Wu3, Wu4, Xia1, Xia2, YanL, YanY}.
Among all the known results, Koch-Tataru proved the well-posedness of $BMO^{-1}$ in 2001.
Giga-Inui-Mahalov-Saal \cite{GIMS} proved the well-posedness of $P^{-1}_{1,1}$ in 2008.
For well-posedness, $BMO^{-1}$ and $P^{-1}_{1,1}$ are the biggest critical  initial spaces in the real version
and in the Fourier transform version respectively.
So we try to consider the above equations in some new point and search special structure solutions.

In 1996, Necas-Ruzicka-Sverak \cite{NRS} proved the only possible self-similar solutions are zero.
If a solution blow up in finite time, then
the initial data and the solution should both have the same special structure.
Hence the study of special structure both initial data and solution have 
will contribute to the study of Navier-Stokes equations \eqref{1.1}.
Further, {\bf Yang-Yang-Wu \cite{YYW} found that the symmetry property plays an important role in the study of ill-posedness.}
In this paper, we search {\bf symmetric and uniform analytic solution}.
Because of the structure of equations \eqref{1.1}, we found {\bf many kinds of symmetric initial data can not produce symmetric solution.}
But we can prove still, if $u_0(x)$ satisfies symmetry property $X_{m}(m=1,2,3)$ and belongs to some critical Fourier Herz space $P^{2-\frac{3}{p}}_{p,q}$,
then there exists global solution $u(t,x)$ such that $u(t,x)$ satisfies still symmetry property $X_{m}(m=1,2,3)$ and
multiply by an exponential growth factor, its uniform norm $U_{e}(\xi)=\sup\limits_{t>0} e^{t^{\frac{1}{2}} |\xi|} | \hat{u}(t,\xi)|$
for $t$  belongs to Herz space $H^{2-\frac{3}{p}}_{p,q}$. Such uniform norm for $t$ means uniform analyticity.

Herz space $H^{\alpha}_{p,q}$ has been studied heavily in harmonic analysis. See Lu-Yang \cite{Lu-Yang} and their references therein.
Phase space $P^{\alpha}_{p,q}$ is defined by the functions whose Fourier
transform belong to the Herz space $H^{\alpha}_{p,q}$. For $p=q$, Herz space becomes the weighted Lebesgue space $W_{\alpha,p}$.
\begin{definition}\label{def1.1}
(i) $f(\xi)\in H^{\alpha}_{p,q}$, if $\{ \sum\limits_{j\in \mathbb{Z}} 2^{qj\alpha} (\displaystyle\int_{2^{j}\leq |\xi|\leq 2^{j+1}} |f(\xi)|^{p}d\xi)^{\frac{q}{p}}\}^{\frac{1}{q}}<\infty$.

(ii) $f(x)\in P^{\alpha}_{p,q}$, if $\hat{f}\in H^{\alpha}_{p,q}$.
\end{definition}
If $\alpha=2-\frac{3}{p}$, then the initial data space $P^{\alpha}_{p,q}$ becomes critical phase spaces.

There are two type of symmetry of a vector field $u$:
(i) symmetry about the component of velocity field.
(ii) symmetry and anti-symmetry for the independent variable of velocity field.
Divergence zero property and non-linearity restrict the possibility of symmetric solutions.
We found only the following two kinds of symmetric initial values can produce symmetric solution up to this paper.
\begin{definition}
(i) $u$ satisfies the symmetry property $X_1$, if
\begin{equation}\label{X1}
\hat{u}_1(\xi_1,\xi_2, \xi_3) = \hat{u}_2(\xi_2, \xi_1,\xi_3).
\end{equation}

(ii) $u$ satisfies the symmetry property $X_2$, if
$\hat{u}$ has the same symmetry as the following polynomial function
\begin{equation} \label{vs}
\left ( \begin{aligned}
\xi_2\xi_3\,&+&\!\! i\xi_1 \\
\xi_1\xi_3\,&+&\!\! i\xi_2\\
\xi_1\xi_2\,&+&\!\! i\xi_3
 \end{aligned}\right )
\end{equation}

(iii) We say $u$ satisfies symmetry property $X_3$, if $u$ satisfies both $X_1$ and $X_2$.

\end{definition}

For equations \eqref{1.1}, we have the following symmetric and uniform analytic solution result:
\begin{theorem}\label{th1.3}
Given $p>\frac{3}{2},1<q<\infty, \alpha=2-\frac{3}{p}$ and $m=1,2,3$.
Given divergence zero $u_0$ satisfies symmetry property $X_{m}$.
There exists constant $C$ such that, if $\|\hat{u}_{0}\|_{(H^{\alpha}_{p,q})^{3}}\leq C$,
then there exists a unique solution $u(t,x)$
satisfies symmetry property $X_{m}$ such that
\begin{equation}\label{11.1}
\sup_{t>0}e^{t^{\frac{1}{2}}|\xi|}|\hat{u}(t, \xi)|\in H^{\alpha}_{p,q}.
\end{equation}
\end{theorem}

We will prove this theorem by constructing symmetric and uniform analytic functions to approach the solution.
{\bf Symmetry property plays an important role in the study of ill-posedness in \cite{YYW}.}
For the two types of symmetry, we note
\begin{remark}
(i) 
As to the symmetry on the component of velocity field,
the symmetry property $X_1$ is a generalization of sub-radial property
which has been considered in Abidi \cite{Abi} and Abidi-Zhang \cite{AbiZ}.
See also Remark \ref{re:2.2}.
$X_1$  can produce the vorticity of the third axi to be zero.

(ii) The independent variables of vector field $u$ have symmetry property means,
the real part and imaginary part of each component of Fourier transform all have axi-symmetry or anti-axi-symmetry.
{\bf There are 262144 kinds of possible symmetry property for vector field.}
Divergence zero and non-linearity restrict the possibility of symmetric solution.
{\bf I have taken a long time to check many kinds of symmetry properties and finally,
we found symmetry property $X_2$ can produce symmetric solution.}
See both subsection \ref{sec:5.1.1} and remark \ref{rem:2.7}.
{\bf I did not found any paper which has considered such symmetric solution.}
Further, I have found in another paper \cite{Yang}, 
there is only one kind of symmetry real valued initial data $u_0(x)$ 
can generate symmetric solution.

\end{remark}

Not considering the symmetry of solution, some people have considered the wellposedness for initial data in Fourier-Herz spaces
$P^{\alpha}_{p,q}$:

\begin{remark} (i)  The following work considered the well-posedness without considering any analyticity:
Cannone-Karch \cite{CK} considered the case $P^{2}_{\infty,\infty}$.
Giga-Inui-Mahalov-Saal \cite{GIMS} and
Lei-Lin \cite{Lei} considered the case $P^{-1}_{1,1}$. Cannone-Wu \cite{CG} considered the case $P^{-1}_{1,q}$. Further,
Xiao-Chen-Fan-Zhou \cite{XCFZ} considered the general case.

(ii) As to the analyticity, see \cite{Ger} and \cite{Miu}. Biswas considered Fourier version weighted Lebesgue space $(W_{\alpha,p})^{3}$
which is the particular Fourier-Herz space $(P^{\alpha}_{ p, p})^{3}$. Applying Theorem 3.2 in \cite{Bis} to the equations \eqref{1.1},
there exists constant $C$ such that $\|u_{0}\|_{(W_{\alpha, p})^{3}}\leq C$, then there exists a unique solution
$u(t, x)\in C([0, \infty), (W_{\alpha , p})^{3})$ satisfying
\begin{align}\label{1.5}
\sup_{0<t<\infty}\|\hat{u}\|_{G_{\theta_{0}}(t)}\leq 2C,
\end{align}
where $\theta_{0}=\frac{3}{p'}-1$ and
\begin{align*}
\|v\|_{G_{\theta}(t)}=\Big\{\displaystyle\int_{\mathcal{R}^{n}}e^{t^{\frac{1}{2 }}p|\xi|}\|\xi|^{\theta p}|v(t, \xi)|^{p}d\xi\Big\}^{\frac{1}{p}}.
\end{align*}
{\bf Since \eqref{11.1} is stronger than \eqref{1.5}, so our result naturally implies the above result.
Further, their methods can not be applied to consider the uniform analyticity.}

(iii) As to the uniform analyticity,  Le Jan-Sznitman \cite{Lej} consider the case $ p=q=\infty$ by  applying convolution inequality.
For  $p=q=2$, which can be found at Chapter 24 in the book of Lemari$\acute{\rm e}$ 2002 \cite{Lem}, whose proof is made by applying paraproduct and space $\dot{B}_{\infty}^{-1, \infty}$.
{\bf In this paper, we can consider uniform analytic solution for general $p$ and $q$. Our main idea is convolution inequality on Herz spaces.}

(iv) {\bf We guess $P^{2-\frac{3}{p}}_{p,\infty}(p>\frac{3}{2})$ is the biggest critical initial data spaces with uniform analytic solution.}
But we know, without considering analyticity, 
$P^{-1}_{1,1}$ in \cite{GIMS} is the biggest initial spaces
among all the critical phase spaces $P^{\alpha}_{p,q}$.
\end{remark}

The structure of the rest paper is as follows: In section 2, we introduce some history on
mild solution and present a result on uniform analytic solution.
In section 3, we show how to transform the wellposedness result to the convolution inequality.
In section 4, we prove the convolution inequality on Herz spaces.
In section 5, we present first how divergence zero exerts influence on symmetry property.
Then we show how the iterative algorithm (\ref{2.6}) inherits the two kinds of symmetry properties.
At the end of section 5, we narrate a corollary on how the symmetry property
reduce the number of induction function in  the iterative algorithm (\ref{2.6}).
In the last section 6, we prove our main theorem on symmetric and uniform analytic solution.

\vskip0.5cm

{\bf Acknowledgement}\ \
(i) Professor Y. Giga has given me some precious suggestions and also told me some development helpful to this paper.
I would like express my sincere thanks to him.

(ii) This paper is financially supported by the
China National Science Foundation (No.11571261).

\section{Mild solution and uniform analyticity}
\setcounter{equation}{0}

The notion of mild solution was pioneered by Kato-Fujita \cite{KatF} in 1960s.
Denote
\begin{align}\label{1.2}
\mathbb{P}\nabla(u\otimes u)=\sum_{l}  \partial _{l}(u_{l}u)-(-\Delta)^{-1} \sum_{l, l'} \partial_{l} \partial_{l'}\nabla(u_{l}u_{l'});
\end{align}
A solution of the above Cauchy problem is then obtained via the integral equation
\begin{equation}\label{1.3}
\begin{split}
u(t, x)&=e^{t\Delta}u_{0}(x)-B(u, u)(t, x);\\
B(u, u)(t, x)&\equiv\displaystyle\int_{0}^{t}e^{ (t-s) \Delta }\mathbb{P}\nabla(u\otimes u)ds,
\end{split}
\end{equation}
which can be solved by a fixed-point method whenever the convergence is suitably defined in some function space.
Denote
\begin{align}\label{1.4}
\begin{array}{cccl}
u^{0}(t,x)&=& e^{t\Delta} u_0\,\, ;& \\
u^{\tau+1}(t,x)&=&  u^{0}(t,x) - B (u^{\tau}, u^{\tau})(t,x),&\forall \tau=0,1,2,\cdots.
\end{array}
\end{align}
For $u_0\in X_0^{3}$, if there exists $X^{3}$ such that $u^{0}(t,x)\in X^{3}$
and iterative quantity $u^{\tau}$ converge to some function $u(t,x)\in X^{3}$, then $u(t,x)$ is the solution of \eqref{1.3}.  $u(t,x)$
is the mild solutions of \eqref{1.1}.
The initial data space $X_0^{3}$  is called to be critical for \eqref{1.1} in the sense that the space is invariant under the scaling
$$
u_{0}^{\lambda}=\lambda u_{0}(\lambda x).
$$
Note that \eqref{1.1} is invariant under the scaling
$$
u_{\lambda}=\lambda u(\lambda^{2}t, \lambda x)
$$
and
$$
p_{\lambda}(t, x)=\lambda^{2}p(\lambda^{2}t, \lambda x).
$$
Hence critical spaces occupied important position.

During the latest
decades, many important results about the mild solutions to \eqref{1.1} have been established; see for example, Cannone \cite{Canj, Canb}, Germain-Pavlovic-
Staffilani \cite{Ger}, Giga-Miyakawa \cite{Gig}, Kato \cite{Kat}, Kato-Ponce \cite{KatP},
Koch-Tataru \cite{Koc}, Taylor \cite{Tay}, Wu \cite{Wu1, Wu2, Wu3, Wu4}.
See also the book \cite{Lem}.
Extending Koch-Tataru's $BMO^{-1}(\mathcal{R}^{n})$ space in \cite{Koc}, Xiao \cite{Xia1, Xia2} introduced the Q-spaces $Q_{0<\alpha<1}^{-1}(\mathcal{R}^{n})$ to investigate the global existence and uniqueness of the classical
Navier-Stokes system. Further, applying the wavelets in \cite{Mey}, we have extended these results to many function spaces: Trieble-Lizorkin spaces,
Besov Morrey spaces and Trieble-Lizorkin Morrey spaces (see \cite{LiX, LiY1, LiY2, Lin, YanL, YanY} ).

Recently, there were also many peoples considered solution in phase spaces. Denote
\begin{equation}\label{fb}
\tilde{B}(\hat{u}, \hat{v})(t, \xi)=\displaystyle\int_{0}^{t}e^{-(t-s)|\xi|^{2 }}\sum_{l}\xi_{l}\hat{u}_{l}*\hat{v}ds
+\xi\displaystyle\int_{0}^{t}e^{-(t-s)|\xi|^{2 }}\sum_{l, l'}\frac{\xi_{l}\xi_{l'}}{|\xi|^{2}}\hat{u}_{l}*\hat{v}_{l'}ds.
\end{equation}
By taking Fourier transform, the above iterative process (\ref{1.4}) is equivalent to the following
\begin{align}\label{2.6}
\hat{u}^{\tau+1}(t, \xi)=e^{-t|\xi|^{2 }}\hat{u}_{0}(\xi)+i\tilde{B}(\hat{u}^{\tau}, \hat{u}^{\tau})(t, \xi), \forall \tau\geq 0.
\end{align}
In this paper, we consider initial value $u_{0}$ belongs to  phase space $(P^{\alpha}_{ p,q})^{3}$.
Without considering analyticity, Giga-Inui-Mahalov-Saal \cite{GIMS} and
Lei-Lin \cite{Lei} obtained a global existence result for small initial data in $P^{-1}_{1, 1}$.
Xiao-Chen-Fan-Zhou considered the general cases.
For uniform analytic solution, only two cases of Fourier-weighted Lebesgue spaces have been considered.
$P^{2}_{\infty, \infty}$ has been studied in Le Jan-Sznitman \cite{Lej} (see also \cite{BisS} and \cite{Lem}).
$P^{\frac{1}{2}}_{2,2}$ has been studied by Fujita-Kato \cite{Fuj} (see also \cite{Lem}).
Fourier-weighted Lebesgue spaces are special Fourier-Herz spaces $P^{\alpha}_{p,p}$.

Among all these spaces,
${\rm BMO}^{-1}$ and $FH^{-1}_{1,1}$ are the biggest critical initial spaces.
So one is interesting to the special structure of solutions.
For $u_0\in X_0^{3}$ or $\hat{u}_0(\xi)\in X_0^{3}$,
\begin{definition}
(i) The solution $u(t,x)$ is analytic means there exists $0<\gamma<1$ such that
\begin{equation}\label{ana}
\sup\limits_{t>0} \|e^{- t^{\gamma} (-\Delta)^{\gamma}} u(t,x)\|_{X_{0}^{3}}<\infty.
\end{equation}

(ii) The solution $u(t,x)$ is uniform analytic means there exists $0<\gamma<1$ such that
\begin{equation}\label{uni-ana1}
\|\sup\limits_{t>0} e^{- t^{\gamma} (-\Delta)^{\gamma}} u(t,x)\|_{X_0^{3}}<\infty,
\end{equation}
or
\begin{equation}\label{uni-ana2}
\|\sup\limits_{t>0} e^{- t^{\gamma} |\xi|^{2\gamma}} \hat{u}(t,\xi)\|_{(X_0)^{3}}<\infty.
\end{equation}
\end{definition}

The earliest study of space analyticity can be found in Masuda \cite{Mas}.
Foias-Teman \cite{Foi} has shown the space analyticity of solution for period function in Sobolev space. Germain-Pavlovi\'{c}-Staffilani \cite{Ger}
constructed the solutions with space analyticity for small initial data in $BMO^{-1}$ (see also \cite{Miu}).
Biswas \cite{Bis} consider Gevrey regularity.
Uniform analyticity (\ref{uni-ana1}) and (\ref{uni-ana2}) mean uniform norm for $t$, which is stronger than analyticity (\ref{ana}).
Analyticity has been considered by many people, but there exist only a few work which considered uniform analyticity.
See Le Jan-Sznitman \cite{Lej} and Lemarie \cite{Lem}. See also \cite{Ger},  \cite{Mas} and \cite{Miu}.
Uniform norm for $t$ means that the dilation does not bring influence to the solution.

For $u(t, x)$, denote its Fourier transform by $\hat{u}(t, \xi)$, its uniform quantity for $t$ denote by $U(\xi)=\sup_{t>0}|\hat{u}(t, \xi)|$ and whose uniform exponential decay quantity denote by
$U_{e}(\xi)=\sup_{t>0}e^{t^{\frac{1}{2}}|\xi|}|\hat{u}(t, \xi)|$.
Herz spaces have been studied heavily in harmonic analysis. See \cite{Lu-Yang}.
We introduce the definition of uniform analyticity in Herz spaces.
\begin{definition}\label{def1.2}
$\hat{u}(t, \xi)\in (S^{\alpha}_{p,q})^3 $, if $U_{e}(\xi)\in H^{\alpha}_{p,q}$.
\end{definition}


In this section, we consider uniform analytic solution in Fourier-Herz spaces.
Our main skills are to prove that the existence of uniform analytic solution is equivalent to the boundedness of convolution inequality on Herz spaces.
\begin{theorem}\label{th3.1}
Given $p>\frac{3}{2}, 1<q<\infty$ and $\alpha=2-\frac{3}{p}$. There exists constant $C$ such that for
$u_0\in P^{\alpha}_{p,q}$ and $\|\hat{u}_{0}\|_{(H^{\alpha}_{p,q})^{3}}\leq C$, there exists a unique solution $u(t,x)$ such that
\begin{align} \label{3.1}
\hat{u}(t, \xi)\in (S^{\alpha}_{p,q})^{3}.
\end{align}
\end{theorem}

Uniform analyticity  in phase space will be considered in this paper by
the estimate of convolution inequality.

\section{From well-posedness to convolution inequality}\label{sec-2}
\setcounter{equation}{0}

\subsection{Bilinear operator $\tilde{\tilde{B}}$ in phase space}
Denote
\begin{align}\label{2.1}
w(t, \xi)\equiv \tilde{\tilde{B}}(u, v)(t, \xi)=|\xi|\displaystyle\int_{0}^{t}e^{-(t-s)|\xi|^{2}}(|u(s, \cdot)|*|v(s, \cdot)|)(\xi)ds
\end{align}



\begin{theorem} \label{th:1}
Given $1\leq p,q\leq \infty$ and $\alpha\in \mathbb{R}$.
\begin{align}\label{2.81}
\tilde{\tilde{B}}\; \text{is bounded from} \;(S^{\alpha}_{p,q})^{3}\times (S^{\alpha}_{p,q})^{3}\;\text{to}\; (S^{\alpha}_{p,q})^{3}
\end{align}
implies
\begin{align}\label{2.8}
\tilde{B}\; \text{is bounded from} \;(S^{\alpha}_{p,q})^{3}\times (S^{\alpha}_{p,q})^{3}\;\text{to}\; (S^{\alpha}_{p,q})^{3}.
\end{align}
\end{theorem}

\begin{proof}
For $k, l, l'\in \{1, 2, 3\}$, denote
$$
\tilde{B}_{l, l'}(\hat{u}_{l}, \hat{v}_{l'})(t, \xi)=\displaystyle\int_{0}^{t}e^{-(t-s)|\xi|^{2 }}\sum_{l}\xi_{l}\hat{u}_{l}*\hat{v}_{l'}ds
$$
and
$$
\tilde{B}_{k, l, l'}(\hat{u}_{l}, \hat{v}_{l'})(t, \xi)
=\displaystyle\int_{0}^{t}e^{-(t-s)|\xi|^{2 }}\sum_{l, l'}\frac{\xi_{k}\xi_{l}\xi_{l'}}{|\xi|^{2}}\hat{u}_{l}*\hat{v}_{l'}ds
$$

We note that, if $\tilde{\tilde{B}}$ satisfies (\ref{2.81}),
then $\tilde{B}_{l, l'}, \tilde{B}_{k, l, l'}$ satisfied
\begin{align}\label{2.10}
\tilde{B}_{l, l'}, \tilde{B}_{k, l, l'}\; \text{is bounded operator from} \;S^{\alpha}_{p,q}\times S^{\alpha}_{p,q}\;\text{to}\; S^{\alpha}_{p,q}.
\end{align}
Further, if $\forall k, l, l'\in \{1, 2, 3\}$, $\tilde{B}_{l, l'}, \tilde{B}_{k, l, l'}$ satisfied the condition \eqref{2.10},
then $\tilde{B}$ satisfied the condition \eqref{2.8}.

\end{proof}


\subsection{From boundedness of $\tilde{\tilde{B}}$ to the convolution inequality}

In this subsection, we transfer the boundedness of bilinear operator $\tilde{\tilde{B}}$ to the estimation of the following convolution inequality on Herz spaces $H^{\alpha}_{p,q}$.

\begin{theorem} \label{th:3.2}

 If $p>\frac{3}{2}, \alpha=2-\frac{2}{p}$ and $1<q<\infty$, then
\begin{align}\label{convolution}
\||\xi|^{-1}(U*V)(\xi)\|_{H^{\alpha}_{p,q}}\leq C\|U\|_{H^{\alpha}_{p,q}}\|V\|_{H^{\alpha}_{p,q}}.
\end{align}
\end{theorem}

\begin{remark} \label{re:3.3}
The weighted Lebesgue cases have been considered in Kerman \cite{Ker}.
\eqref{convolution} extend such results to Herz spaces.
\end{remark}

In the rest of this subsection, we show how the above convolution inequality implies the boundedness of $\tilde{\tilde{B}}$.
\begin{theorem}\label{th:2}
Given $p>\frac{3}{2}, \alpha=2-\frac{2}{p}$ and $1<q<\infty$.
If the inequality \eqref{convolution} is true, then
$\tilde{\tilde{B}}$ is bounded form $S^{\alpha}_{p,q}\times S^{\alpha}_{p,q}$ to $S^{\alpha}_{p,q}$.
\end{theorem}

\begin{proof}
We claim the following inequality
\begin{align}\label{2.13}
e^{-(t-s)|\xi|^{2 }}e^{-s^{\frac{1}{2 }}|\xi-\eta|}e^{-s^{\frac{1}{2 }}|\eta|}\leq
e^{2}e^{-t^{\frac{1}{2 }}|\xi|}e^{-\frac{1}{2}(t-s)|\xi|^{2 }}.
\end{align}
In fact, using the triangle inequality gives
\begin{align*}
e^{-(t-s)|\xi|^{2 }}e^{-s^{\frac{1}{2 }}|\xi-\eta|}e^{-s^{\frac{1}{2 }}|\eta|}\leq
e^{-\frac{1}{2}(t-s)|\xi|^{2 }+(t^{\frac{1}{2 }}-s^{\frac{1}{2 }})|\xi|}e^{-t^{\frac{1}{2 }}|\xi|}e^{-\frac{1}{2}(t-s)|\xi|^{2 }}.
\end{align*}
Let $I=-\frac{1}{2}(t-s)|\xi|^{2 }+(t^{\frac{1}{2 }}-s^{\frac{1}{2 }})|\xi|$, it suffices to prove
$$
I\leq 2 .
$$
Note that
\begin{align*}
x_{2}^{2 }-x_{1}^{2 }=(x_{2}-x_{1})(x_{2}+ x_{1}).
\end{align*}
Hence, we have
\begin{align*}
I=(t^{\frac{1}{2 }}-s^{\frac{1}{2 }})|\xi|\Big(1-\frac{1}{2} (t^{\frac{1}{2 }} + s^{\frac{1}{2 }})|\xi|\Big).
\end{align*}
On one hand, if $(t^{\frac{1}{2 }}+ s^{\frac{1}{2 }})|\xi|\geq 2 $, then we have $I\leq 0<2 $. On the other hand,
if $(t^{\frac{1}{2 }}+ s^{\frac{1}{2 }})|\xi|\leq 2  $, then we have $I\leq t^{\frac{1}{2 }}|\xi|\leq 2$.

Denote $w_{e}(t, \xi)=e^{t^{\frac{1}{2 }}|\xi|}w(t, \xi)$, $u_{e}(t, \xi)=e^{t^{\frac{1}{2 }}|\xi|}u(t, \xi)$ and
$v_{e}(t, \xi)=e^{t^{\frac{1}{2 }}|\xi|}v(t, \xi)$.
Hence \eqref{2.1} and \eqref{2.13} imply the following inequality:
\begin{align}\label{2.14}
w_{e}(t, \xi)\leq e^{2  }|\xi|\displaystyle\int_{0}^{t}e^{-\frac{1}{2}(t-s)|\xi|^{2 }}ds(|u_{e}(t, \cdot)|*|v_{e}(t, \cdot)|)(\xi)ds.
\end{align}
Let $W_{e}(\xi)=\sup_{t>0}|w_{e}(t, \xi)|$, $U_{e}(\xi)=\sup_{t>0}|u_{e}(t, \xi)|$ and $V_{e}(\xi)=\sup_{t>0}|v_{e}(t, \xi)|$. Taking supremum for $s$ and $t$ in the
equation \eqref{2.1}, we get
\begin{align*}
W_{e}(\xi)\leq C|\xi|\sup_{t>0}\displaystyle\int_{0}^{t}e^{-\frac{1}{2}(t-s)|\xi|^{2 }}ds(U_{e}*V_{e})(\xi)\leq C|\xi|^{-1 }(U_{e}*V_{e})(\xi).
\end{align*}
Hence the convolution inequality \eqref{convolution} implies the boundedness of $\tilde{\tilde{B}}$ form $S^{\alpha}_{p,q}\times S^{\alpha}_{p,q}$ to $S^{\alpha}_{p,q}$.
\end{proof}

\subsection{Proof of theorem \ref{th3.1}}
In this subsection, we prove theorem \ref{th3.1}.
For a distribution $u_{0}$, Picard's iterative process is to find out a mild solution near some function $u^{0}$.
See Cannone \cite{Canb} and Koch-Tataru \cite{Koc};
see also Li-Yang \cite{LiY1, LiY2} and Lin-Yang \cite{Lin}.

\begin{proof}
For small initial value $u_{0}$ in some space $X_{0}^{3}$, if there exists certain space $X^{3}$ such that
\begin{align}\label{2.4}
\text{the operator}\; e^{t \Delta  }\; \text{is continuous from} \;X_{0}^{3}\;\text{to}\; X^{3}
\end{align}
and
\begin{align}\label{2.5}
B\; \text{is bounded operator from} \;X^{3}\times X^{3}\;\text{to}\; X^{3},
\end{align}
then $B(u^{\tau}, u^{\tau})(t, x)$ defined in the equation (\ref{1.4}) is an error term. Picard's contraction principle tells that $u^{\tau}$ converges to
the unique solution of the above Navier-Stokes equation \eqref{1.1}.

Take $X_{0}^{3}=(P^{\alpha}_{p,q})^{3}$ and $X^{3}=(S^{\alpha}_{p,q})^{3}$.
It is easy to see
\begin{theorem} \label{th:3}
Given $1\leq p, q\leq \infty, \alpha\in \mathbb{R}$.
$e^{t \Delta  }$ is continuous from $(P^{\alpha}_{p,q})^{3}$ to $(S^{\alpha}_{p,q})^{3}$. 
\end{theorem}
By theorems \ref{th:1}, \ref{th:3.2}, \ref{th:2} and \ref{th:3}, we know Theorem \ref{th3.1} is true.
\end{proof}


\section{The proof of convolution inequality}\label{sec:main}
\setcounter{equation}{0}

To simplify the proof, we prove only the case $\frac{3}{2}<p<3$ in this section.
For $p\geq 3$, we have to use the duality of Herz spaces.
For $1<p,q<\infty, \alpha\in \mathbb{R}$, we have $(H^{\alpha}_{p,q})'= H^{-\alpha}_{p',q'}$.
If we use auxiliary function $h\in H^{-\alpha}_{p',q'}$ and the similar idea in this section, we can prove that
$\int |\xi|^{-1} |U(\xi)| \ast |V(\xi)| |h(\xi)||\xi|^{-\alpha} d\xi$ is bounded,
and the above result is true for all $p>\frac{3}{2}$.

Denote $W_{U,V}= \sum\limits_{j\in \mathbb{Z}} 2^{qj(1-\frac{3}{p})}\{\int_{2^{j}\leq |\xi|\leq 2^{j+1}} (|U(\xi)|\ast |V(\xi)|)^p d\xi\}^{\frac{q}{p}}.$ We decompose the integration of $\xi$ to the dyadic ring and get
$$W_{U,V}\leq C\sum\limits_{j\in \mathbb{Z}} 2^{qj(1-\frac{3}{p})} \{\int_{2^{j}\leq |\xi|\leq 2^{j+1}} (\int |U(\xi-\eta)| |V(\eta)|d\eta)^{p} d\xi\}^{\frac{q}{p}}.$$

We decompose then the integration of $\eta$ to three parts and get
$$\begin{array}{rl}
W_{U,V}\leq &C\sum\limits_{j\in \mathbb{Z}} 2^{qj(1-\frac{3}{p})} \{\int_{2^{j}\leq |\xi|\leq 2^{j+1}} (\int_{ |\eta|\leq 2^{j-1}} |U(\xi-\eta)| |V(\eta)|d\eta)^{p} d\xi \}^{\frac{q}{p}}\\
&+C\sum\limits_{j\in \mathbb{Z}} 2^{qj(1-\frac{3}{p})} \{\int_{2^{j}\leq |\xi|\leq 2^{j+1}} (\int_{ |\eta|\geq 2^{j+2}} |U(\xi-\eta)| |V(\eta)|d\eta)^{p} d\xi\}^{\frac{q}{p}}\\
&+C\sum\limits_{j\in \mathbb{Z}}  2^{qj(1-\frac{3}{p})}\{\int_{2^{j}\leq |\xi|\leq 2^{j+1}} (\int_{2^{j-1}\leq |\eta|\leq 2^{j+2}} |U(\xi-\eta)| |V(\eta)|d\eta)^{p} d\xi \}^{\frac{q}{p}}\\
\equiv & I_1+ I_2 +I_3.\end{array}$$
For $I_1, I_2$ and $I_3$, we apply different way to enlarge the convolution part and for $i=1,2,3$, we prove $I_i\leq C\|U\|_{H^{\alpha}_{p,q}}^q \|V\|_{H^{\alpha}_{p,q}}^q$. In this section, for $1<p\leq \infty$, denote $p'=\frac{p}{p-1}$.

For $I_1$, we have
$$\begin{array}{rl}
I_1\leq & C\sum\limits_{j\in \mathbb{Z}} 2^{qj(1-\frac{3}{p})} \{\int_{2^{j}\leq |\xi|\leq 2^{j+1}} \int_{ |\eta|\leq 2^{j-1}} |U(\xi-\eta)|^p |V(\eta)|d\eta  d\xi (\int_{ |\eta|\leq 2^{j-1}}  |V(\eta)|d\eta)^{p-1}
\}^{\frac{q}{p}}\\
\leq & C\sum\limits_{j\in \mathbb{Z}} 2^{qj(1-\frac{3}{p})} \{ \int_{2^{j-1}\leq |\xi|\leq 5\cdot 2^{j}}  |U(\xi)|^p  d\xi (\int_{ |\eta|\leq 2^{j-1}}  |V(\eta)|d\eta)^{p}\}^{\frac{q}{p}}\\
=& C\sum\limits_{j\in \mathbb{Z}} 2^{qj(1-\frac{3}{p})} \{ \int_{2^{j-1}\leq |\xi|\leq 5\cdot 2^{j}}  |U(\xi)|^p  d\xi \}^{\frac{q}{p}} (\int_{ |\eta|\leq 2^{j-1}}  |V(\eta)|d\eta)^{q}.
\end{array}$$
Further,
$$\begin{array}{rl}
\int_{ |\eta|\leq 2^{j-1}}  |V(\eta)|d\eta &= \sum\limits_{j'\leq j-2} \int_{ 2^{j'}\leq |\eta|\leq 2^{j'+1}}  |V(\eta)|d\eta \\
&\leq C \sum\limits_{j'\leq j-2}  2^{\frac{3j'}{p'}} \{\int_{ 2^{j'}\leq |\eta|\leq 2^{j'+1}}  |V(\eta)|^{p}d\eta\} ^{\frac{1}{p}} \\
&=  C \sum\limits_{j'\leq j-2} 2^{j'\alpha} \{\int_{ 2^{j'}\leq |\eta|\leq 2^{j'+1}}  |V(\eta)|^{p}d\eta\} ^{\frac{1}{p}}  2^{j'} \\
&\leq C (\sum\limits_{j'\leq j-2} 2^{qj'\alpha} \{\int_{ 2^{j'}\leq |\eta|\leq 2^{j'+1}}  |V(\eta)|^{p}d\eta\} ^{\frac{q}{p}})^{\frac{1}{q}} (\sum\limits_{j'\leq j-2} 2^{q'j'})^{\frac{1}{q'}} \\
&\leq 2^{j} (\sum\limits_{j'\leq j-2} 2^{qj'\alpha} \{\int_{ 2^{j'}\leq |\eta|\leq 2^{j'+1}}  |V(\eta)|^{p}d\eta\} ^{\frac{q}{p}})^{\frac{1}{q}}.
\end{array}$$
Hence
$$I_1\leq C\|U\|_{H^{\alpha}_{p,q}}^q \|V\|_{H^{\alpha}_{p,q}}^q.$$

Then we consider $I_3$.
Take $\alpha< \lambda<\frac{3}{p'}$.
$$\begin{array}{rl} &I_{U,V}=\int_{2^{j}\leq |\xi|\leq 2^{j+1}} (\int_{2^{j-1}\leq |\eta|\leq 2^{j+2}} |U(\xi-\eta)| |V(\eta)|d\eta)^{p} d\xi  \\
\leq &\int_{2^{j}\leq |\xi|\leq 2^{j+1}} \int_{2^{j-1}\leq |\eta|\leq 2^{j+2}} |\xi-\eta|^{\lambda p} |U(\xi-\eta)|^{p} |V(\eta)| ^{p} d\eta  (\int_{2^{j-1}\leq |\eta|\leq 2^{j+2}}|\xi-\eta|^{-\lambda p'} d\eta)^{p-1} d\xi  \\
\leq &\int_{2^{j}\leq |\xi|\leq 2^{j+1}} \int_{2^{j-1}\leq |\eta|\leq 2^{j+2}} |\xi-\eta|^{\lambda p} |U(\xi-\eta)|^{p} |V(\eta)| ^{p} d\eta  (\int_{ |\xi-\eta|\leq 3\cdot 2^{j+1}}|\xi-\eta|^{-\lambda p'} d\eta)^{p-1} d\xi  \\
\leq & 2^{ (3-\lambda p') (p-1) j} \int_{2^{j}\leq |\xi|\leq 2^{j+1}} \int_{2^{j-1}\leq |\eta|\leq 2^{j+2}} |\xi-\eta|^{\lambda p} |U(\xi-\eta)|^{p} |V(\eta)| ^{p} d\eta  d\xi  \\
\leq & 2^{ (3-\lambda p') (p-1) j}  \int_{ |\xi |\leq 2^{j+3}} |\xi|^{\lambda p} |U(\xi)|^{p}   d\xi  \int_{2^{j-1}\leq |\eta|\leq 2^{j+2}}   |V(\eta)| ^{p} d\eta.\end{array}$$

(i) For $q> p$,  we have
$$\begin{array}{rl}& \int_{ |\xi |\leq 2^{j+3}} |\xi|^{\lambda p} |U(\xi)|^{p}   d\xi= \sum\limits_{j'\leq j+2} \int_{ 2^{j'}\leq  |\xi |\leq 2^{j'+1}} |\xi|^{\lambda p} |U(\xi)|^{p}   d\xi   \\
\leq & C \sum\limits_{j'\leq j+2} 2^{p \lambda j'}\int_{ 2^{j'}\leq  |\xi |\leq 2^{j'+1}} |U(\xi)|^{p}   d\xi\\
=&C \sum\limits_{j'\leq j+2} 2^{p (\lambda-\alpha)j'} 2^{p \alpha j'}\int_{ 2^{j'}\leq  |\xi |\leq 2^{j'+1}} |U(\xi)|^{p}   d\xi\\
\leq & C \{ \sum\limits_{j'\leq j+2} 2^{q\alpha j'} (\int_{ 2^{j'}\leq  |\xi |\leq 2^{j'+1}} |U(\xi)|^{p}   d\xi)^{\frac{q}{p} } \} ^{\frac{p}{q}}
\{ \sum\limits_{j'\leq j+2} 2^{\frac{pq}{q-p}(\lambda-\alpha )j'} \} ^{\frac{q-p}{q}}\\
\leq & 2^{pj -(3-\lambda p') (p-1) j}  \{\sum\limits_{j'\leq j+2} 2^{qj'\alpha} (\displaystyle\int_{2^{j'}\leq |\xi|\leq 2^{j'+1}} |U(\xi)|^{p}d\xi)^{\frac{q}{p}}\}^{\frac{p}{q}}.\end{array}$$
Hence
$$\begin{array}{l} I_{U,V}\leq 2^{pj} \{\sum\limits_{j'\leq j+2} 2^{qj'\alpha} (\displaystyle\int_{2^{j'}\leq |\xi|\leq 2^{j'+1}} |U(\xi)|^{p}d\xi)^{\frac{q}{p}}\}^{\frac{p}{q}}
\int_{2^{j-1}\leq |\eta|\leq 2^{j+2}}  |V(\eta)|^{p} d\eta  .\end{array}$$

So we get
$$ I_3 \leq  C
\sum\limits_{j\in \mathbb{Z}} 2^{ q\alpha j}
\sum\limits_{j'\leq j+2} 2^{qj'\alpha} (\displaystyle\int_{2^{j'}\leq |\xi|\leq 2^{j'+1}} |U(\xi)|^{p}d\xi)^{\frac{q}{p}}
\{\int_{ 2^{j-1}\leq |\eta|\leq 2^{j+2}}  |V(\eta)|^p d\eta \}^{\frac{q}{p}}. $$
Hence
$$I_3\leq C\|U\|_{H^{\alpha}_{p,q}}^q \|V\|_{H^{\alpha}_{p,q}}^q.$$

(ii) For $q \leq  p$,  we have
$$\begin{array}{rl} & \{ \int_{ |\xi |\leq 2^{j+3}} |\xi|^{\lambda p} |U(\xi)|^{p}   d\xi \}^{\frac{q}{p}}\\
= &\{\sum\limits_{j'\leq j+2} \int_{ 2^{j'}\leq  |\xi |\leq 2^{j'+1}} |\xi|^{\lambda p} |U(\xi)|^{p}   d\xi \}^{\frac{q}{p}}  \\
\leq  &C \sum\limits_{j'\leq j+2} 2^{q \lambda j'} \{\int_{ 2^{j'}\leq  |\xi |\leq 2^{j'+1}} |U(\xi)|^{p}   d\xi \}^{\frac{q}{p}}.
\end{array}$$

Hence
$$\begin{array}{rcl} \{I_{U,V} \}^{\frac{q}{p}}
&\leq & 2^{qj(\frac{3}{p'}-\lambda) } \sum\limits_{j'\leq j+2} 2^{qj'\lambda} (\displaystyle\int_{2^{j'}\leq |\xi|\leq 2^{j'+1}} |U(\xi)|^{p}d\xi)^{\frac{q}{p}}
\{\int_{2^{j-1}\leq |\eta|\leq 2^{j+2}}  |V(\eta)|^{p} d\eta\}^{\frac{q}{p}}\\
&\leq & 2^{qj(\frac{3}{p'}-\alpha) } \sum\limits_{j'\leq j+2} 2^{qj'\alpha} (\displaystyle\int_{2^{j'}\leq |\xi|\leq 2^{j'+1}} |U(\xi)|^{p}d\xi)^{\frac{q}{p}}
\{\int_{2^{j-1}\leq |\eta|\leq 2^{j+2}}  |V(\eta)|^{p} d\eta\}^{\frac{q}{p}}  .\end{array}$$

Hence we get
$$ I_3 \leq  C
\sum\limits_{j\in \mathbb{Z}} 2^{ q\alpha j}
\sum\limits_{j'\leq j+2} 2^{qj'\alpha} (\displaystyle\int_{2^{j'}\leq |\xi|\leq 2^{j'+1}} |U(\xi)|^{p}d\xi)^{\frac{q}{p}}
\{\int_{ 2^{j-1}\leq |\eta|\leq 2^{j+2}}  |V(\eta)|^p d\eta \}^{\frac{q}{p}}. $$
Hence
$$I_3\leq C\|U\|_{H^{\alpha}_{p,q}}^q \|V\|_{H^{\alpha}_{p,q}}^q.$$

In the end of this section, we estimate $I_2$.
Take $\rho, \delta,\delta'$ satisfying that $\frac{p}{3}<\rho \leq \min (1, p-1), 0<\delta'< \delta<\frac{3\rho}{p}-1 $. We have
$$\begin{array}{rcl}
I_2&\leq & \sum\limits_{j\in \mathbb{Z}} 2^{qj(1-\frac{3}{p})}\{ \int_{2^{j}\leq |\xi|\leq 2^{j+1}} \sum\limits_{\tau\geq 0} 2^{ p \tau \delta'}
(\int _{2^{j+2+\tau}\leq |\eta|\leq 2^{j+3+\tau}} |U(\xi-\eta)|| V(\eta)| d\eta)^{p} d\xi\}^{\frac{q}{p}} \\
&\leq & \sum\limits_{j\in \mathbb{Z},\tau\geq 0} 2^{qj(1-\frac{3}{p})+ q\delta \tau}\{ \int_{2^{j}\leq |\xi|\leq 2^{j+1}}
(\int _{2^{j+2+\tau}\leq |\eta|\leq 2^{j+3+\tau}} |U(\xi-\eta)|| V(\eta)| d\eta)^{p} d\xi\}^{\frac{q}{p}} \\
&= & \sum\limits_{j\in \mathbb{Z},\tau\geq 0} 2^{qj(1-\frac{3}{p})+ q\delta \tau}\{ \int_{2^{j}\leq |\xi|\leq 2^{j+1}}
(\int _{2^{j+2+\tau}\leq |\eta|\leq 2^{j+3+\tau}} |U(\xi-\eta)|^{1-\rho} | V(\eta)| |U(\xi-\eta)|^{\rho} d\eta)^{p} d\xi\}^{\frac{q}{p}} \\
&\leq &\sum\limits_{j\in \mathbb{Z},\tau\geq 0} 2^{ qj(1-\frac{3}{p}) + q\tau\delta } \{\int_{2^{j}\leq |\xi|\leq 2^{j+1},2^{j+2+\tau}\leq |\eta|\leq 2^{j+3+\tau}}
|U(\xi-\eta) |^{p(1-\rho)} |V(\eta)|^p d\eta d\xi \}^{\frac{q}{p}}\\
&\times & \sup\limits_{2^{j}\leq |\xi|\leq 2^{j+1}}
\{ \int _{2^{j+2+\tau}\leq |\eta|\leq 2^{j+3+\tau}} |U(\xi-\eta)| ^{p'\rho} d\eta \}^{\frac{q}{p'}}.
\end{array}$$

Since we have
$$\begin{array}{rl}
&\int_{2^{j}\leq |\xi|\leq 2^{j+1} } |U(\xi-\eta) |^{p(1-\rho)} d\xi\\
\leq & \{\int_{2^{j}\leq |\xi|\leq 2^{j+1} } |U(\xi-\eta) |^{p} d\xi\}^{1-\rho}
\{ \int_{2^{j}\leq |\xi|\leq 2^{j+1} } d\xi \}^{\rho}\\
\leq & C 2^{3j\rho} \{\int_{2^{j}\leq |\xi|\leq 2^{j+1} } |U(\xi-\eta) |^{p} d\xi\}^{1-\rho},
\end{array}$$ hence
$$\begin{array}{rcl}
I_2
&\leq &\sum\limits_{j\in \mathbb{Z},\tau\geq 0} 2^{qj(1-\frac{3}{p})+\frac{3q\rho j}{p} + q\tau\delta }
(\int_{2^{j+\tau+2} -2^{j+1}\leq |\xi|\leq 2^{j+\tau+3}+ 2^{j+1}} | U(\xi) |^{p' \rho} d\xi)^{\frac{q}{p'}}\\
 &\times & \{ \int_{2^{j+2+\tau}\leq |\eta|\leq 2^{j+3+\tau}} |V(\eta)|^p d\eta
(\int_{2^{j}\leq |\xi|\leq 2^{j+1}}
|U(\xi-\eta)|^p d\xi)^{1-\rho} \}^{\frac{q}{p}} .
\end{array}$$

Since $2^{j+2+\tau}\leq |\eta|\leq 2^{j+3+\tau}$, we enlarge the integration region of $U(\xi-\eta)$ for the variable $\xi$
$$\int_{2^{j}\leq |\xi|\leq 2^{j+1}}
|U(\xi-\eta)|^p d\xi\leq \int_{2^{j+\tau+2} -2^{j+1}\leq |\xi|\leq 2^{j+\tau+3}+ 2^{j+1}}
|U(\xi)|^p d\xi.$$

Since $\rho \leq p-1$, applying Cauchy-Schwartz inequality, we get
$$\int_{2^{j+\tau+2} -2^{j+1}\leq |\xi|\leq 2^{j+\tau+3}+ 2^{j+1}}
| U(\xi)|^{p'\rho}  d\xi\leq C(\int_{2^{j+\tau+2} -2^{j+1}\leq |\xi|\leq 2^{j+\tau+3}+ 2^{j+1}}
| U(\xi)|^p  d\xi)^{\frac{p'\rho }{p}} 2^{3(j+\tau)(1-\frac{p'\rho}{p})}.$$
Hence
$$\begin{array}{rcl}
I_2 &\leq & \sum\limits_{j\in \mathbb{Z},\tau\geq 0} 2^{qj(1-\frac{3}{p})+ q\tau\delta+\frac{3q\rho}{p}j+ 3q(j+\tau) (\frac{1}{p'}-\frac{\rho}{p}) }
\{\int_{2^{j+\tau+2} -2^{j+1}\leq |\xi|\leq 2^{j+\tau+3}+ 2^{j+1}}  |U(\xi)|^p  d\xi\}^{\frac{q}{p}} \\
&\times &\{\int_{2^{j+2+\tau}\leq |\eta|\leq 2^{j+3+\tau}} |V(\eta)|^p d\eta \}^{\frac{q}{p}} \\
&= &\sum\limits_{j\in \mathbb{Z},\tau\geq 0} 2^{qj(2-\frac{3}{p})+q(j+\tau)(2-\frac{3}{p})+q\tau(1+\delta-\frac{3\rho}{p}) }
\{\int_{2^{j+\tau+2} -2^{j+1}\leq |\xi|\leq 2^{j+\tau+3}+ 2^{j+1}}  |U(\xi)|^p  d\xi \}^{\frac{q}{p}}\\
&\times & \{\int_{2^{j+2+\tau}\leq |\eta|\leq 2^{j+3+\tau}} |V(\eta)|^p d\eta \}^{\frac{q}{p}} .
\end{array}$$

For $ \frac{p}{3} <\rho \leq \min (1, p-1)$ and $0<p\delta \leq 3\rho-p$, we have
$$I_2\leq C \|U\|^{q}_{H^{\alpha}_{p,q}}\|V\|^{q}_{H^{\alpha}_{p,q}}.$$



\section{Symmetry property}
\setcounter{equation}{0}

In this paper, the main important structure of solution is the symmetry property of solution.
The divergence zero property and non-linearity property restrict the possibility of symmetry.
Symmetry property plays an important role in the study of ill-posedness in \cite{YYW}.
If the initial value has certain symmetry, whether the solution has still the same symmetry?

\subsection{Divergence zero and iterative algorithm}
\subsubsection{Divergence zero.}\label{sec:5.1.1}
We notice that divergence zero implies
\begin{lemma}
If $\nabla(u_1, u_2, u_3)=0$, then
\begin{equation}\label{ediv}\hat{u}_3 = \frac {- (\xi_1 \hat{u}_{1}  +\xi_2 \hat{u}_{2})}{\xi_3}.
\end{equation}
\end{lemma}

(i) For compressible Navier-Stokes equations, there exists radial solutions. The incompressible Navier-Stokes equations have zero divergence property, so there is not radial solution. But there exists still sub-radial solution. Abidi \cite{Abi} and Abidi-Zhang \cite{AbiZ} identifies such particular symmetry with a general name axi-symmetry. Sub-radial is just a particular case of
the symmetry about the component of velocity field in this paper. See also remark \ref{re:2.2}.
The velocity field $u$ satisfies the symmetry property $X_1$
can produce that the vorticity for the third axi is zero.

(ii) As to the symmetry relative to the independent variable of velocity field, we note,
for a scale function $f(\xi)= a(\xi)+i b(\xi)$, $f(\xi)$ has symmetry property,
if its real part and imaginary part all  have axi-symmetry or anti-axi-symmetry.
There are 64 kinds of symmetry property for each scale function.
We can label each possibility with polynomial function $\xi^{\alpha}+ i \xi^{\beta}$,
where $\alpha=(\alpha_1, \alpha_2, \alpha_3), \beta=(\beta_1, \beta_2, \beta_3)\in \{0,1\}^{3}$.
We write them in the following way
$$Tf= Ta+ iTb= T\xi^{\alpha}+i T\xi^{\beta}.$$
For $\sigma=1,2,3$,
$\alpha_\sigma=0$  denotes function $a$ has symmetry for variable $\xi_\sigma$.
$\alpha_\sigma=1$  denotes function $a$ has antisymmetry for variable $\xi_\sigma$.
The same thing for $\beta$.

A vector function has $64^{3}=262144$ kinds of such symmetry, but by \eqref{ediv},
a velocity field $u$ with divergence zero property has $64^2$ kinds of symmetry property at most.
We found only that  the symmetry property $X_2$ can be preserved by the recursion formula (\ref{2.6})
up to this paper.
In another paper \cite{Yang}, I have found
there is only one kind of symmetry real valued initial data $u_0(x)$
can generate symmetric solution.

For vector field $u(x)$, if $u$ satisfies the symmetry property $X_2$, then
\begin{equation} \label{X2}
T(\hat{u}) = \left( \begin{aligned}
T(\xi_2\xi_3)\,&+&\!\! iT\xi_1 \\
T(\xi_1\xi_3)\,&+&\!\! iT\xi_2\\
T(\xi_1\xi_2)\,&+&\!\! iT\xi_3
 \end{aligned}\right )
\end{equation}

\subsection{Iterative algorithm}
To consider the action of non-linearity to symmetry,
we rewrite the quantities in (\ref{fb}) and reformulate (\ref{2.6}).
$\forall k=1,2,3$, denote
$$A_{k} (\hat{u}^{\tau},\hat{u}^{\tau}) = \sum\limits_{l} \xi_l (\hat{u}_{l}^{\tau}\ast \hat{u}_{k}^{\tau}),$$
Denote further
$$\begin{array}{c}A_{0} (\hat{u}^{\tau},\hat{u}^{\tau}) = \sum\limits_{l}  \sum\limits_{l'} \xi_l \xi_{l'} (\hat{u}_{l}^{\tau}\ast \hat{u}_{l'}^{\tau})\\
= \xi_1^2 (\hat{u}_1^{\tau}\ast \hat{u}_1^{\tau}) + 2\xi_1\xi_2 (\hat{u}_1^{\tau} \ast \hat{u}_2^{\tau})+ \xi_2^2 (\hat{u}_2^{\tau}\ast \hat{u}_2^{\tau})
+ 2\xi_1\xi_3 (\hat{u}_1^{\tau}\ast \hat{u}_3^{\tau})+2\xi_2\xi_3 (\hat{u}_2\ast \hat{u}_3^{\tau})+ \xi_3^2 (\hat{u}_3^{\tau}\ast \hat{u}_3^{\tau}).
\end{array}$$

Then equation (\ref{2.6}) can be written as following:
\begin{equation}\label{eq:5.25}
\hat{u}_k^{\tau+1}(t,\xi)= e^{-t\xi^2} \hat{u}_{0,k}(\xi)+ i \int^{t}_{0} e^{-(t-s)\xi^2} A_{k}(\hat{u}^{\tau}, \hat{u}^{\tau})ds
-\frac{i\xi_k}{\xi^2} \int^{t}_{0} e^{-(t-s)\xi^2} A_{0}(\hat{u}^{\tau}, \hat{u}^{\tau}) ds.
\end{equation}

In this section, we consider certain symmetry of iterative algorithm \eqref{eq:5.25} and reduction of iterative function for solution.
The first kind of symmetry is the symmetry of the component of velocity field, which is the generalization of Abidi-Zhang's sub-radial property.
The iterative algorithm (\ref{2.6}) has the heritability for such symmetry.
The second symmetry is the symmetry of coordinate variable.

\subsection{Symmetry with respect to the component of velocity field}
For compressible Navier-Stokes equations, we can consider radial solution. But for incompressible Navier-Stokes equations, because of zero divergence property, there exists not radial solution.
But many person have consider sub-radial solution with cylindrical coordinates.
Let $x=(x_1,x_2,x_3)\in \mathbb{R}^{3}$.
denote the cylindrical coordinates of $x$ by $(r,\theta, x_3)$ where
$r(x_1,x_2)  { \tiny \begin{array}{c} {\mbox { def}} \\=\end{array}} \sqrt{x_1^2+ x_2^2}$ and
$\theta(x_1, x_2) { \tiny \begin{array}{c} {\mbox { def}} \\=\end{array}} \arctan \frac{x_2}{x_1}$
with $r\in [0,\infty), \theta\in [0,2\pi]$. Denote
$e_r { \tiny \begin{array}{c} {\mbox { def}} \\=\end{array}}  (  \cos \theta, \sin\theta, 0), $
$e_\theta { \tiny \begin{array}{c} {\mbox { def}} \\=\end{array}}  ( -\sin\theta,  \cos \theta, 0) $.
Abidi and Zhang said a function is axisymmetry, if
$$u(t,x_1, x_2,x_3)= u_{r}(t,r,x_3) e_r+ u_3 (t,r,x_3) (0,0,1).$$
Such axisymmetry is called to be also sub-radial case.

\begin{remark}\label{re:2.2} The property (\ref{X1}) is more general than subradial case.
In fact,

(i) If one takes Fourier transform for $(\cos\theta u_{r}(t,r,x_3), \sin\theta u_{r}(t,r,x_3), u_3 (t,r,x_3))$ in the ordinary axi system $(x_1,x_2,x_3)$, we can find that subradial case satisfies property (\ref{X1}).

(ii) For $\tau\in\mathbb{R}_{+}$,
we take $$\hat{u}(t, \xi_1, \xi_2,\xi_3) = (\xi_1 g(t, |\xi_{1}|^{\tau} + |\xi_2|^{\tau}, \xi_3), \xi_2 g(t, |\xi_{1}|^{\tau} + |\xi_2|^{\tau}, \xi_3), h(t, \xi_1,\xi_2,\xi_3)),$$
where $h$ is obtained by zero divergence property (\ref{ediv}). Evidently, $u(t,x)$ satisfies property (\ref{X1}), but if $\tau\neq 2$, it is not a subradial function.
\end{remark}

In this paper, we consider the heritability of property (\ref{X1}).

\begin{theorem} \label{S1}
If $u^{0}$ satisfies (\ref{X1}), then $\forall \tau\geq 0$, $u^{\tau+1}$ satisfies (\ref{X1}).
\end{theorem}

By applying zero divergence property, we regroup the term in $A_{k}(\hat{u}^{\tau}, \hat{u}^{\tau})(k=0,1,2,3)$.
$\forall k=1,2$, we divide each $A_{k}(\hat{u}^{\tau}, \hat{u}^{\tau})$ into two terms:
$$A^{1}_{k}(\hat{u}^{\tau}, \hat{u}^{\tau})= \sum\limits_{l=1,2} \xi_l (\hat{u}_{l}^{\tau}\ast \hat{u}_{k}^{\tau}),$$
$$A^{2}_{k}(\hat{u}^{\tau}, \hat{u}^{\tau})= -\xi_3\sum\limits_{l=1,2} \int \frac{\xi_{l}-\eta_l}{\xi_3-\eta_3}  (\hat{u}_{l}^{\tau}(\xi-\eta) \ast \hat{u}_{k}^{\tau})(\eta) d\eta .$$

Further,
$$\begin{array}{rl}
A_{0} (\hat{u}^{\tau}, \hat{u}^{\tau})=\!\! & \!\!\xi_{1}^{2} (\hat{u}_{1}^{\tau} \ast \hat{u}_{1}^{\tau})+ \xi_{2}^{2} (\hat{u}_{2}^{\tau} \ast \hat{u}_{2}^{\tau})+2\xi_{1}\xi_{2} (\hat{u}_{1}^{\tau} \ast \hat{u}_{2}^{\tau})\\
&-2\xi_1\xi_{3}  \int \sum\limits_{l=1,2} \frac{\xi_l-\eta_l} {\xi_3-\eta_3} \hat{u}_l^{\tau}(\xi-\eta)\hat{u}_{1}(\eta) d\eta
-2 \xi_{2} \xi_3 \int \sum\limits_{l=1,2} \frac{\xi_l-\eta_l} {\xi_3-\eta_3} \hat{u}_l^{\tau}(\xi-\eta)\hat{u}_{2}(\eta) d\eta \\
&+ \xi_3^2  \int \sum\limits_{l,l'=1,2} \frac{\xi_l-\eta_l} {\xi_3-\eta_3} \hat{\eta_{l'}}{\eta_3} \hat{u}_l^{\tau}(\xi-\eta)\hat{u}_{l'}(\eta) d\eta.
\end{array}$$
We divide $A_{0}(\hat{u}^{\tau}, \hat{u}^{\tau})$ into six terms:
$$A^{1}_{0,0} (\hat{u}^{\tau}, \hat{u}^{\tau})= \xi_{1}^{2} (\hat{u}_{1}^{\tau} \ast \hat{u}_{1}^{\tau})+ \xi_{2}^{2} (\hat{u}_{2}^{\tau} \ast \hat{u}_{2}^{\tau}),$$
$$A^{1}_{0,1} (\hat{u}^{\tau}, \hat{u}^{\tau})= 2\xi_{1}\xi_{2} (\hat{u}_{1}^{\tau} \ast \hat{u}_{2}^{\tau}),$$
$$A^{2}_{0,0} (\hat{u}^{\tau}, \hat{u}^{\tau})= -2\xi_{3} \{\xi_{1} \int \frac{\xi_1-\eta_1} {\xi_3-\eta_3} \hat{u}_1^{\tau}(\xi-\eta)\hat{u}_{1}(\eta) d\eta
+\xi_{2} \int \frac{\xi_2-\eta_2} {\xi_3-\eta_3} \hat{u}_2^{\tau}(\xi-\eta)\hat{u}_{2}(\eta) d\eta\},  $$
$$A^{2}_{0,1} (\hat{u}^{\tau}, \hat{u}^{\tau})= -2\xi_{3} \{\xi_{1} \int \frac{\xi_2-\eta_2} {\xi_3-\eta_3} \hat{u}_2^{\tau}(\xi-\eta)\hat{u}_{1}(\eta) d\eta
+\xi_{2} \int \frac{\xi_1-\eta_1} {\xi_3-\eta_3} \hat{u}_1^{\tau}(\xi-\eta)\hat{u}_{2}(\eta) d\eta\},  $$
$$A^{3}_{0,0} (\hat{u}^{\tau}, \hat{u}^{\tau})= \xi_{3}^2 \{  \int \frac{\xi_1-\eta_1} {\xi_3-\eta_3}\frac{\eta_1}{\eta_3} \hat{u}_1^{\tau}(\xi-\eta)\hat{u}_{1}(\eta) d\eta
+  \int \frac{\xi_2-\eta_2} {\xi_3-\eta_3} \frac{\eta_2}{\eta_3} \hat{u}_2^{\tau}(\xi-\eta)\hat{u}_{2}(\eta) d\eta\},  $$
$$A^{3}_{0,1} (\hat{u}^{\tau}, \hat{u}^{\tau})= \xi_{3}^2 \{  \int \frac{\xi_1-\eta_1} {\xi_3-\eta_3}\frac{\eta_2}{\eta_3} \hat{u}_1^{\tau}(\xi-\eta)\hat{u}_{2}(\eta) d\eta
+  \int \frac{\xi_2-\eta_2} {\xi_3-\eta_3} \frac{\eta_1}{\eta_3} \hat{u}_2^{\tau}(\xi-\eta)\hat{u}_{1}(\eta) d\eta\}.  $$
Hence,
$$\begin{array}{rl}
& A_{0} (\hat{u}^{\tau}, \hat{u}^{\tau})\\
\equiv \!\!&\!\! A^{1}_{0,0} (\hat{u}^{\tau}, \hat{u}^{\tau})+ A^{1}_{0,1} (\hat{u}^{\tau}, \hat{u}^{\tau})+ A^{2}_{0,0} (\hat{u}^{\tau}, \hat{u}^{\tau})+ A^{2}_{0,1} (\hat{u}^{\tau}, \hat{u}^{\tau})+ A^{3}_{0,0} (\hat{u}^{\tau}, \hat{u}^{\tau})+ A^{3}_{0,1} (\hat{u}^{\tau}, \hat{u}^{\tau}).
\end{array}$$

(i) For the first two terms $A^{1}_{2}(\hat{u}^{\tau}, \hat{u}^{\tau})(\xi_1,\xi_2,\xi_3)$ and $A^{2}_{2}(\hat{u}^{\tau}, \hat{u}^{\tau})(\xi_1,\xi_2,\xi_3)$, we have
\begin{lemma}\label{S1.1}
If $u^{\tau}$ satisfies (\ref{X1}), then
\begin{equation}\label{eq:S1.1} A^{1}_{2}(\hat{u}^{\tau}, \hat{u}^{\tau})(\xi_2,\xi_1,\xi_3)= A^{1}_{1}(\hat{u}^{\tau}, \hat{u}^{\tau})(\xi_1,\xi_2,\xi_3).
\end{equation}
\begin{equation}\label{eq:S1.2}A^{2}_{2}(\hat{u}^{\tau}, \hat{u}^{\tau})(\xi_2,\xi_1,\xi_3)= A^{2}_{1}(\hat{u}^{\tau}, \hat{u}^{\tau})(\xi_1,\xi_2,\xi_3).
\end{equation}
\end{lemma}

\begin{proof} First, we have
$$\begin{array}{rl}
&A^{1}_{2}(\hat{u}^{\tau}, \hat{u}^{\tau})(\xi_2,\xi_1,\xi_3)\\
= & \xi_2\int \hat{u}_1^{\tau}(\xi_2-\eta_1, \xi_1-\eta_2, \xi_3-\eta_3) \hat{u}_1^{\tau} (\eta_1,\eta_2, \eta_3) d\eta\\
&+ \xi_1\int \hat{u}_2^{\tau}(\xi_2-\eta_1, \xi_1-\eta_2, \xi_3-\eta_3) \hat{u}_1^{\tau} (\eta_1,\eta_2, \eta_3) d\eta\\
= &\xi_2\int \hat{u}_2^{\tau}(\xi_1-\eta_2, \xi_2-\eta_1, \xi_3-\eta_3) \hat{u}_1^{\tau} (\eta_2,\eta_1, \eta_3) d\eta\\
&+ \xi_2\int \hat{u}_1^{\tau}(\xi_1-\eta_2, \xi_2-\eta_1, \xi_3-\eta_3) \hat{u}_1^{\tau} (\eta_2,\eta_1, \eta_3) d\eta\\
=&A^{1}_{1}(\hat{u}^{\tau}, \hat{u}^{\tau})(\xi_1,\xi_2,\xi_3).
\end{array}$$
Further, we have
$$\begin{array}{rl}
&A^{2}_{2}(\hat{u}^{\tau}, \hat{u}^{\tau})(\xi_2,\xi_1,\xi_3)\\
= &-\xi_3 \int \frac{\xi_2-\eta_1}{\xi_3-\eta_3} \hat{u}_1^{\tau}(\xi_2-\eta_1, \xi_1-\eta_2, \xi_3-\eta_3) \hat{u}_2^{\tau} (\eta_1,\eta_2, \eta_3) d\eta\\
&- \xi_3\int \frac{\xi_1-\eta_2}{\xi_3-\eta_3}\hat{u}_2^{\tau}(\xi_2-\eta_1, \xi_1-\eta_2, \xi_3-\eta_3) \hat{u}_2^{\tau} (\eta_1,\eta_2, \eta_3) d\eta\\
=& -\xi_3\int \frac{\xi_2-\eta_1}{\xi_3-\eta_3}\hat{u}_2^{\tau}(\xi_1-\eta_2, \xi_2-\eta_1, \xi_3-\eta_3) \hat{u}_1^{\tau} (\eta_2,\eta_1, \eta_3) d\eta\\
&- \xi_3 \int \frac{\xi_1-\eta_2}{\xi_3-\eta_3}\hat{u}_1^{\tau}(\xi_1-\eta_1, \xi_2-\eta_2, \xi_3-\eta_3) \hat{u}_1^{\tau} (\eta_2,\eta_1, \eta_3) d\eta\\
=&A^{2}_{1}(\hat{u}^{\tau}, \hat{u}^{\tau})(\xi_1,\xi_2,\xi_3).
\end{array}$$
\end{proof}

(ii) For the rest six terms, we have
\begin{lemma} \label{S1.2}
If $u^{\tau}$ satisfies (\ref{X1}), then for $k=1,2,3$ and $l=0,1$, we have
\begin{equation}\label{eq:S1.3}A^{k}_{0,l}(\hat{u}^{\tau}, \hat{u}^{\tau})(\xi_2,\xi_1,\xi_3) = A^{k}_{0,l}(\hat{u}^{\tau}, \hat{u}^{\tau})(\xi_1,\xi_2,\xi_3).
\end{equation}
\end{lemma}

\begin{proof}
For the six term of $A_{0}(\hat{u}^{\tau}, \hat{u}^{\tau})(\xi_2,\xi_1,\xi_3)$, we consider only $A^{1}_{0,0}(\hat{u}^{\tau}, \hat{u}^{\tau})$. The other five term $A^{1}_{0,1}(\hat{u}^{\tau}, \hat{u}^{\tau}), A^{2}_{0,0}(\hat{u}^{\tau}, \hat{u}^{\tau}), A^{2}_{0,1}(\hat{u}^{\tau}, \hat{u}^{\tau}),
A^{3}_{0,0}(\hat{u}^{\tau}, \hat{u}^{\tau})$ and $A^{3}_{0,1}(\hat{u}^{\tau}, \hat{u}^{\tau})$ can be verified similarly.

For $A^{1}_{0,0}(\hat{u}^{\tau}, \hat{u}^{\tau})$, we have
$$\begin{array}{rl}
&A^{1}_{0,0}(\hat{u}^{\tau}, \hat{u}^{\tau})(\xi_2,\xi_1,\xi_3)\\
=& \xi_2^2 \int  \hat{u}_1^{\tau}(\xi_2-\eta_1, \xi_1-\eta_2, \xi_3-\eta_3) \hat{u}_1^{\tau} (\eta_1,\eta_2, \eta_3) d\eta\\
&+ \xi_1^2 \int  \hat{u}_2^{\tau}(\xi_2-\eta_1, \xi_1-\eta_2, \xi_3-\eta_3) \hat{u}_2^{\tau} (\eta_1,\eta_2, \eta_3) d\eta\\
= &\xi_2^2 \int  \hat{u}_2^{\tau}(\xi_1-\eta_2, \xi_2-\eta_1, \xi_3-\eta_3) \hat{u}_2^{\tau} (\eta_2,\eta_1, \eta_3) d\eta\\
&+ \xi_1^2 \int  \hat{u}_1^{\tau}(\xi_1-\eta_2, \xi_2-\eta_1, \xi_3-\eta_3) \hat{u}_1^{\tau} (\eta_2,\eta_1, \eta_3) d\eta.
\end{array}$$
Change $\eta_1$ and $\eta_2$ in the above last equality, we get that
$$A^{1}_{0,0}(\hat{u}^{\tau}, \hat{u}^{\tau})(\xi_2,\xi_1,\xi_3)= A^{1}_{0,0}(\hat{u}^{\tau}, \hat{u}^{\tau})(\xi_1,\xi_2,\xi_3).$$
\end{proof}

Now we prove theorem \ref{S1}.
\begin{proof} By lemmas \ref{S1.1} and \ref{S1.2},
$\forall \tau\geq 0$, if $u^{\tau}$ satisfies (\ref{X1}), then the eight terms satisfies equations (\ref{eq:S1.1}), (\ref{eq:S1.2}) and (\ref{eq:S1.3}).
Hence $u^{\tau+1}$ satisfies (\ref{X1}).
\end{proof}

\subsection{Symmetry of independent variable}

For a complex function $f$, if one of its real part or its imaginary part has no axi-symmetry property, then we say that $f$ has no symmetry property.
For two function $f_1$ and $f_2$, if and only if $f_1$ and $f_2$ have the same symmetry, we can consider the symmetry of the sum of $f_1$ and $f_2$. Further, $T(f_1+f_2)= Tf_1=Tf_2$.

For the product of two real function $f$ and $g$, it is easy to see
\begin{lemma} Given $\alpha, \beta\in \{0,1\}^{3}.$
If $Tf=T\xi^{\alpha}$ and $Tg=T\xi^{\beta}$, then
$$T(fg)= T\xi^{\alpha+\beta},$$
where $\alpha_k+\beta_k \, (k=1,2,3) $, modulate 2, belongs to $\{0,1\}$.
\end{lemma}

For the convolution of two real functions $f$ and $g$,
\begin{lemma}  Given $\alpha, \beta\in \{0,1\}^{3}.$
If $Tf=T\xi^{\alpha}$ and $Tg=T\xi^{\beta}$, then $T(f*g)= T\xi^{\alpha+\beta}$.
\end{lemma}
\begin{proof}
We consider only the axi-symmetry for the first coordinate axis, the other case can be proved in a similar way.
That is to say, if $f(\xi_1, \xi_2,\xi_3)=f(-\xi_1, \xi_2, \xi_3)$ and $g(\xi_1, \xi_2,\xi_3)=g(-\xi_1, \xi_2, \xi_3)$, then
$f\ast g(\xi_1, \xi_2,\xi_3)=f\ast g(-\xi_1, \xi_2, \xi_3)$.
$$\begin{array}{rcl}
f\ast g(-\xi_1, \xi_2, \xi_3) &=& \int^{\infty}_{-\infty} \int_{\mathbb{R}^2} f(-\xi_1-\eta_1, \xi_2-\eta_2, \xi_3-\eta_3) g(\eta_1,\eta_2,\eta_3) d\eta_1 d\eta_2 d\eta_3\\
{\mbox { symmetry of }} g &=& \int^{\infty}_{-\infty} \int_{\mathbb{R}^2} f(-\xi_1-\eta_1, \xi_2-\eta_2, \xi_3-\eta_3) g(-\eta_1,\eta_2,\eta_3) d\eta_1 d\eta_2 d\eta_3 \\
-\eta_1\rightarrow \eta_1 &=& \int^{\infty}_{-\infty} \int_{\mathbb{R}^2} f(-\xi_1+\eta_1, \xi_2-\eta_2, \xi_3-\eta_3) g(\eta_1,\eta_2,\eta_3) d\eta_1 d\eta_2 d\eta_3\\
{\mbox{ symmetry of }} f &=& \int^{\infty}_{-\infty} \int_{\mathbb{R}^2} f(\xi_1-\eta_1, \xi_2-\eta_2, \xi_3-\eta_3) g(\eta_1,\eta_2,\eta_3) d\eta_1 d\eta_2 d\eta_3 \\
&=&f\ast g(\xi_1, \xi_2, \xi_3).
\end{array} $$
\end{proof}

We get down to the second main result on symmetry.

\begin{theorem} \label{S2}
If $u^{0}$ satisfies (\ref{X2}), then $\forall \tau\geq 0$, $u^{\tau+1}$ satisfies (\ref{X2}).
\end{theorem}

\begin{remark} \label{rem:2.7}
Non-linearity may change the symmetry property.
Many kinds of symmetry of complex vector field
can not be preserved in the iterative algorithm (\ref{2.6}).
For example,
a component has three anti-symmetry. Denote \\
$$T\hat{u}_0= \left(\begin{array}{c} T\xi_2\\
T\xi_1\\ T(\xi_1\xi_2\xi_3)\end{array}\right)$$
We can prove that $u^{0}$ and $B(u^{0}, u^{0})$ have different symmetry property.
\end{remark}

Before the proof of Theorem \ref{S2}, we introduce some notations and a lemma.
$\forall k=1,2,3$, denote $u_{k}^{\tau}= a_{k}^{\tau}+i b_{k}^{\tau}$, then the real part of $ A_{k}(\hat{u}^{\tau}, \hat{u}^{\tau})(\xi)$ can be written as:
$$ A_{k,1}(\hat{u}^{\tau}, \hat{u}^{\tau})(\xi)= \sum\limits_{l} \xi_{l} \{ a_l^{\tau}\ast b_k^{\tau} + b_{l}^{\tau}\ast a^{\tau}_{k}\}.$$
The imaginary part of $ A_{k}(\hat{u}^{\tau}, \hat{u}^{\tau})(\xi)$ can be written as:
$$ A_{k,2}(\hat{u}^{\tau}, \hat{u}^{\tau})(\xi)= \sum\limits_{l} \xi_{l} \{a_l^{\tau}\ast a_k^{\tau} - b_{l}^{\tau}\ast b^{\tau}_{k}\}.$$
The real part of $ A_{0}(\hat{u}^{\tau}, \hat{u}^{\tau})(\xi)$ can be written as the sum of the following two terms:
$$ A_{0,1}(\hat{u}^{\tau}, \hat{u}^{\tau})(\xi)= \sum\limits_{l}\sum\limits_{l'} \xi_{l} \xi_{l'} a_l^{\tau}\ast b_{l'}^{\tau}  $$ and
$$ A_{0,2}(\hat{u}^{\tau}, \hat{u}^{\tau})(\xi)= \sum\limits_{l} \sum\limits_{l'} \xi_{l} \xi_{l'}  b_{l}^{\tau} \ast a_{l'}^{\tau} .$$
The imaginary part of $ A_{k}(\hat{u}^{\tau}, \hat{u}^{\tau})(\xi)$ can be written as the sum of the following two terms:
$$ A_{0,3}(\hat{u}^{\tau}, \hat{u}^{\tau})(\xi)= \sum\limits_{l}\sum\limits_{l'} \xi_{l} \xi_{l'} a_l^{\tau}\ast a_{l'}^{\tau}  $$ and
$$ A_{0,4}(\hat{u}^{\tau}, \hat{u}^{\tau})(\xi)= - \sum\limits_{l} \sum\limits_{l'} \xi_{l} \xi_{l'}  b_{l}^{\tau} \ast b_{l'}^{\tau} .$$

Hence we can write the  equation (\ref{2.6}) as follows:
\begin{equation}
\begin{array}{rl}
\hat{u}_k^{\tau+1}(t,\xi)  = & e^{-t\xi^2} \hat{u}_{0,k}(\xi)- \int^{t}_{0} e^{-(t-s)\xi^2}  A_{k,1}(\hat{u}^{\tau}, \hat{u}^{\tau})(\xi)  ds\\
&+ i \int^{t}_{0} e^{-(t-s)\xi^2}  A_{k,2}(\hat{u}^{\tau}, \hat{u}^{\tau})(\xi) ds\\
& -\frac{\xi_k}{\xi^2} \int^{t}_{0} e^{-(t-s)\xi^2} \{ A_{0,1}(\hat{u}^{\tau}, \hat{u}^{\tau})(\xi)+  A_{0,2}(\hat{u}^{\tau}, \hat{u}^{\tau})(\xi)\} ds\\
&+\frac{i\xi_k}{\xi^2} \int^{t}_{0} e^{-(t-s)\xi^2} \{ A_{0,3}(\hat{u}^{\tau}, \hat{u}^{\tau})(\xi)+  A_{0,4}(\hat{u}^{\tau}, \hat{u}^{\tau})(\xi)\} ds.
\end{array}
\end{equation}

We prove first
\begin{lemma} \label{A0}
$$TA_0(\hat{u}^{\tau},\hat{u}^{\tau}) = T(\xi_1\xi_2\xi_3) + i T1.$$
\end{lemma}

\begin{proof}
There exists four term in $A_0(\hat{u}^{\tau},\hat{u}^{\tau})$. We consider first two real part term:
$$\begin{array}{rl}
&T(A_{0,1}(\xi))\\
=& T\xi_{1}^2 T(\xi_2\xi_3) T\xi_1 + T\xi_2^{2} T(\xi_1 \xi_3) T\xi_2 + T\xi_{3}^2 T(\xi_1 \xi_2) T\xi_3 + T(\xi_1 \xi_2) T(\xi_2 \xi_3) T\xi_2\\
+& T(\xi_1\xi_3) T(\xi_2 \xi_3)T\xi_3 + T(\xi_2\xi_3)T(\xi_1\xi_3)T\xi_3\\
=& T(\xi_1\xi_2\xi_3).
\end{array}$$
And,
$$\begin{array}{rl}
&T(A_{0,2}(\xi))\\
=& T\xi_{1}^2 T(\xi_1) T (\xi_2 \xi_3)  + T\xi_2^{2} T\xi_2 T(\xi_1\xi_3) + T\xi_{3}^2 T\xi_3 T(\xi_1\xi_2) + T(\xi_1 \xi_2) T\xi_1 T(\xi_1\xi_3)\\
+& T(\xi_1\xi_3) T\xi_1 T(\xi_1\xi_2) + T(\xi_2\xi_3)T \xi_3 T(\xi_1\xi_2)\\
=& T(\xi_1\xi_2\xi_3).
\end{array}$$

Further, we consider the two imaginary part term:
$$\begin{array}{rl}
&T(A_{0,3}(\xi))\\
=& T\xi_{1}^2 T(\xi_2\xi_3) T(\xi_2\xi_3) + T\xi_2^{2} T(\xi_1 \xi_3) T(\xi_1\xi_3) + T\xi_{3}^2 T(\xi_1 \xi_2) T(\xi_1\xi_2)\\
+ & T(\xi_1 \xi_2) T(\xi_2 \xi_3) T(\xi_2\xi_3) + T(\xi_1\xi_3) T(\xi_2 \xi_3)T(\xi_1\xi_2) + T(\xi_2\xi_3)T(\xi_1\xi_3)T(\xi_1\xi_2)\\
=& T1.
\end{array}$$
and
$$\begin{array}{rl}
&T(A_{0,4}(\xi))\\
=& T\xi_{1}^2 T(\xi_1) T \xi_1  + T\xi_2^{2} T\xi_2 T\xi_2 + T\xi_{3}^2 T\xi_3 T\xi_3\\
+ & T(\xi_1 \xi_2) T\xi_1 T\xi_2 + T(\xi_1\xi_3) T\xi_1 T\xi_3 + T(\xi_2\xi_3)T \xi_2 T\xi_3\\
=& T1.
\end{array}$$
\end{proof}

Now we come to prove the main theorem of this subsection.
\begin{proof} First, we have
\begin{equation}\label{111}
\begin{array}{rl}
&T(A_{1,1}(\xi))\\
=& T\xi_{1} \{T(\xi_2 \xi_3) T\xi_1 +  T\xi_1 T(\xi_2\xi_3) \} + T\xi_2 \{T(\xi_1\xi_3) T\xi_1 +T \xi_2 T(\xi_2\xi_3)\}\\
&+ T\xi_3 \{T(\xi_1\xi_2) T\xi_1+ T\xi_3 T(\xi_2\xi_3)\}\\
= &T(\xi_2 \xi_3).
\end{array}\end{equation}
and
\begin{equation}\label{112}
\begin{array}{rl}
&T(A_{2,1}(\xi))\\
=& T\xi_{1} \{T(\xi_2 \xi_3) T\xi_2 +  T\xi_1 T(\xi_1\xi_3) \} + T\xi_2 \{T(\xi_1\xi_3) T\xi_2 +T \xi_2 T(\xi_1\xi_3)\}\\
&+ T\xi_3 \{T(\xi_1\xi_2) T\xi_2+ T\xi_3 T(\xi_1\xi_3)\}\\
= &T(\xi_1 \xi_3).
\end{array}
\end{equation}
Further, by divergence zero property, we have
\begin{equation}\label{113}
T(A_{3,1}(\xi))= T(\xi_1 \xi_2).
\end{equation}

Secondly, we have
\begin{equation}\label{121}\begin{array}{rl}
&T(A_{1,2}(\xi))\\
=& T\xi_{1} \{T(\xi_2 \xi_3) T(\xi_2\xi_3) +  T\xi_1 T \xi_1) \} + T\xi_2 \{T(\xi_1\xi_3) T(\xi_2\xi_3) +T \xi_2 T\xi_1)\}\\
&+ T\xi_3 \{T(\xi_1\xi_2) T(\xi_2\xi_3)+ T\xi_3 T\xi_1 \}\\
= & T\xi_1
\end{array}
\end{equation}
and
\begin{equation}\label{122}
\begin{array}{rl}
&T(A_{2,2}(\xi))\\
=& T\xi_{1} \{T(\xi_2 \xi_3) T(\xi_1\xi_3) +  T\xi_1 T \xi_3) \} + T\xi_2 \{T(\xi_1\xi_3) T(\xi_1\xi_3) +T \xi_2 T\xi_2)\}\\
&+ T\xi_3 \{T(\xi_1\xi_2) T(\xi_1\xi_3)+ T\xi_3 T\xi_2 \}\\
= &T\xi_2.
\end{array}
\end{equation}
Further, by divergence zero property, we have
\begin{equation}\label{123}
T(A_{3,2}(\xi))= T\xi_3.
\end{equation}

Together the above symmetry (\ref{111}, \ref {112}, \ref{113}, \ref{121}, \ref{122}, \ref{123}),  with the symmetry in the lemma \ref{A0}, we get the symmetry property of Theorem \ref{S2}.

\end{proof}


\subsection{Reduction of iteration function}
For a complex initial data,
the iterative algorithm (\ref{2.6}) is an iteration of four real functions.
By Theorems \ref{S1} and \ref{S2}, if we consider the symmetry property,
the number of iterative function will be reduced.

\begin{cor} \label{cro2.1}
(i) If the initial value $u_0$ satisfies symmetry property (\ref{X1}), then the iterative algorithm (\ref{2.6}) can be reduce to the iteration of a complex function.

(ii) If the initial value $u_0$ satisfies symmetry property (\ref{X2}) and the real part of $\hat{u}_{0}$ is zero, then the iterative algorithm (\ref{2.6}) can be reduce to the iteration of two real functions.

(iii) If the initial value $u_0$ satisfies symmetry properties (\ref{X1})  and (\ref{X2}) and the real part of $\hat{u}_{0}$ is zero, then the iterative algorithm (\ref{2.6}) can be reduce to the iteration of one real functions.

\end{cor}


\section{Proof of the main theorem}
\setcounter{equation}{0}

The proof of the main theorem is the corollary of the five Theorems \ref{th:1}, \ref{th:3.2}, \ref{th:2}, \ref{S1} and \ref{S2}.
In fact, we have
\begin{proof}
We prove the main theorem \ref{th1.3} by considering the iterative process (\ref{2.6}).
We search functions with symmetry property $X_{m}$ and uniform analyticity property to approach the solution.
First, if $u_0$ is a function with symmetry property $X_{m}$, then $e^{t\Delta}u_0$ is a function with symmetry property $X_{m}$ and
with uniform analyticity property.
According to theorems \ref{S1} and \ref{S2}, $u^{\tau}(\tau\geq 1)$ are functions with symmetry property $X_{m}$.
According to theorems \ref{th:1}, \ref{th:3.2} and \ref{th:2},
$u^{\tau}(\tau\geq 1)$ are functions with uniform analyticity property.

Since $u_0$ is a small initial value in $(P^{\alpha}_{p,q})^{3}$, by theorems \ref{th:1}, \ref{th:3.2} and \ref{th:2},
$u^{\tau}$ convergence to the solution $u(t,x)$ and $u(t,x)$ is a function with symmetry property $X_{m}$ and with uniform analyticity property.

\end{proof}

\end{document}